\documentclass[10pt,a4paper,notitlepage, oneside]{article}
\usepackage[utf8x]{inputenc} 
\usepackage{ucs} 
\usepackage[english]{babel}
\usepackage{amsmath}
\usepackage{amsthm}
\usepackage{amsfonts}
\usepackage{amssymb}
\usepackage{mathtools}
\usepackage{cite}
\usepackage[percent]{overpic}
\usepackage{hyperref}
\hypersetup{
 colorlinks=true,  
 linkcolor=black,    
 urlcolor=red      
}

\usepackage{enumerate} 

\usepackage{tikz}
\usetikzlibrary{backgrounds}

\usepackage{todonotes}  

\newtheorem{definition}{Definition}

\newtheorem{theorem}[definition]{Theorem}
\newtheorem{corollary}[definition]{Corollary}
\newtheorem{lemma}[definition]{Lemma}

\newcommand{\emtext}[1]{\text{\em #1}}

\DeclareMathOperator{\sign}{sgn}

\newcommand{\Nz}{\ensuremath{\mathbb{N}}} 
\newcommand{\Zz}{\ensuremath{\mathbb{Z}}} 

\title{Jacobsthal numbers in generalised Petersen graphs}
\author{Henning Bruhn \and Laura Gellert \and Jacob G\"unther}
\date{}
\begin{document}

\maketitle 
\begin{abstract}
We prove that the number of $1$-factorisations of a generalised Petersen
graph of the type $GP(3k,k)$ is equal to the $k$th Jacobsthal number $J(k)$
if $k$ is odd, and equal to $4J(k)$, when $k$ is even. 
Moreover, we verify the list colouring conjecture for $GP(3k,k)$. 
\end{abstract}

\section{Introduction}

Often, combinatorial objects that on the surface seem quite different
nevertheless exhibit a deeper, somewhat hidden, connection. This is, 
for instance, the case for 
tilings of $3 \times (k-1)$-rectangles with $1 \times 1$ and $2\times 2$-squares~\cite{Heu99},
certain meets in lattices~\cite{Day_Kleit_West79},
and
the number of walks of length~$k$ between adjacent vertices in a triangle~\cite{Barry07}:
in all three cases the cardinality is equal to the $k$th \emph{Jacobsthal number}. 
Their sequence
\[
0,1,1,3,5,11,21,43,85,171,341\ldots
\]
is defined by the recurrence relation $J(k)= J(k-1) + 2 J(k-2)$
and initial values $J(0)=0$ and $J(1)=1$. Jacobsthal numbers also appear in the context
of 
alternating sign matrices~\cite{Frey_Sel00},
the Collatz problem and in the study of 
necktie knots~\cite{FM00}; 
see~\cite[A001045]{oeis} for much more.

In this article,  we add to this list by describing a relationship to 
certain \emph{generalised Petersen graphs} $GP(3k,k)$. 
These graphs arise from matching $k$ disjoint triangles to 
triples of equidistant vertices on a cycle of $3k$ vertices; 
see below for a precise definition 
and Figure~\ref{duererfig} for two examples.

\begin{theorem} \label{Thm_Number_1fact}
For odd $k$, 
the number of $1$-factorisations of the generalised Petersen graph $GP(3k,k)$ equals the Jacobsthal number $J(k)$; for even $k$, the number is equal to~$4J(k)$.
\end{theorem}

A \emph{$1$-factorisation} of a graph $G=(V,E)$ is a partition of the edge set into 
perfect matchings.
(A \emph{perfect matching} is a 
set of $^{|V|}/_2$ edges, no two of which share an endvertex.) 
Such factorisations are closely linked to edge colourings: indeed, a $d$-regular graph $G$
has a $1$-factorisation if and only if its edge set can be coloured with $d$ colours. 
That is, the \emph{chromatic index}, the minimal number of colours needed to colour
all the edges, is equal to~$d$. 
\begin{figure}[ht]
\centering
\def\drawGP{
\def\angle{360/\vxnumber}
\foreach \i in {1,...,\vxnumber}{
    \draw[hedge] (\angle*\i+90:\radius) -- (\angle*\i+\angle+90:\radius); 
    \draw[hedge] (\angle*\i+90:\radius) -- (\angle*\i+90:\innerradius);
    \draw[hedge] (\angle*\i+90:\innerradius) -- (\angle*\i+\angle*\jumpnumber+90:\innerradius);
}
\foreach \i in {1,...,\vxnumber}{
    \node[hvertex] (v\i) at (\angle*\i+90:\radius) {};
    \node[hvertex] (ui\i) at (\angle*\i+90:\innerradius) {};
}
}
\begin{tikzpicture}[every edge quotes/.style={},scale=1]
\tikzstyle{hvertex}=[thick,circle,inner sep=0.cm, minimum size=2.5mm, fill=white, draw=black]
\tikzstyle{hedge}=[ultra thick]

\begin{scope}
\def\vxnumber{6}      
\def\jumpnumber{2}    
\def\radius{1.8cm}
\def\innerradius{1cm}
\drawGP
\end{scope}

\begin{scope}[shift={(5,0)}]
\def\vxnumber{9}      
\def\jumpnumber{3}    
\def\radius{1.8cm}
\def\innerradius{1cm}
\drawGP
\end{scope}

\end{tikzpicture}
\caption{The D\"urer graph $GP(6,2)$ and the generalised Petersen graph $GP(9,3)$}
\label{duererfig}
\end{figure}
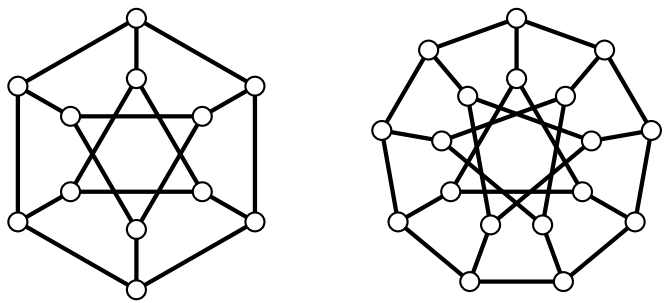

\emph{List edge-colourings} generalise edge colourings. 
Given lists $L_e$ of allowed colours at every edge $e \in E$, 
the task consists in colouring the edges so that every edge $e$ receives a colour 
from its list $L_e$. 
The \emph{choice index} of $G$ is the smallest number $\ell$ 
so that any collection of lists $L_e$ of size~$\ell$ each 
allows a list colouring. 
The choice index is at least as large as the chromatic index. 
The famous \emph{list-colouring conjecture} asserts that the two indices never differ:

\newtheorem*{lcc}{List-colouring conjecture}

\begin{lcc}
The chromatic index of every simple graph equals its choice index.
\end{lcc}

While the conjecture has been verified for some graph classes, 
 bipartite graphs~\cite{Gal95} and regular planar graphs~\cite{EG96}
for instance,
the conjecture remains wide open for most 
graph classes, among them cubic graphs.  
We prove:

\begin{theorem} \label{Thm_LECC_proved}
The list-colouring conjecture is true for generalised Petersen graphs $GP(3k,k)$.\sloppy
\end{theorem}

Our proof is based on the algebraic colouring criterion of Alon and Tarsi~\cite{A_Tar92}.
In our setting, it suffices to check that, for a suitable definition of 
a sign, the number of positive $1$-factorisations differs from the number 
of negative $1$-factorisations. In this respect our second topic ties in 
quite nicely with our first, and we will be able to re-use some of
the observations leading to Theorem~\ref{Thm_Number_1fact}. 

Generalised Petersen graphs were first studied by Coxeter~\cite{Cox50}. 
For $k, n \in \Nz$ with $k < \tfrac{n}{2}$,
the graph $GP(n,k)$  is defined as the graph on  vertex set $\{ u_i, v_i \, : \, i \in \Zz_{n}  \}$ 
with edge set $\{ u_iu_{i+1}, u_iv_i, v_iv_{i+k} \, : \, i~\in~\Zz_{n} \}$. 
Generalised Petersen graphs are cubic graphs. All of them, except the Petersen graph 
itself, have chromatic index~$3$; see Watkins~\cite{Watkins69}, and
Castagna and Prins~\cite{Cas_Prins72}. In particular, this means that 
the list colouring conjecture for them does not follow from the list version 
of Brooks' theorem. 
We focus in this article on the graphs $GP(3k,k)$, the smallest
of which, $GP(6,2)$, is called the \emph{Dürer graph}.  


\smallskip

We follow standard graph theory notation as can be found, for instance, in 
the book of Diestel~\cite{diestelBook10}.

\section{Counting $1$-factorisations}\label{Section_Counting_1-fact}

In the rest of the article we consider a fixed generalised Petersen graph $GP=GP(3k,k)$.
The \emph{outer cycle} $C_O$ of $GP$, the cycle $u_0u_1\ldots u_{3k-1}u_{0}$, will play a 
key role. Its edges we call \emph{outer edges}, while the edges $u_iv_i$ for $i\in\mathbb Z_{3k}$ are called \emph{spokes}. See Figure~\ref{Fig_Vertices_V_i} for an illustration.

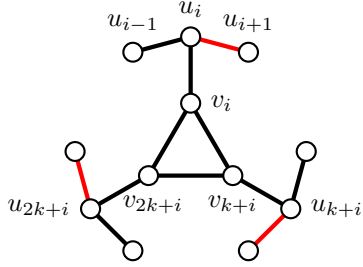
\begin{figure}[ht]
\centering
\begin{tikzpicture}[scale=0.8,auto]
\tikzstyle{hvertex}=[thick,circle,inner sep=0.cm, minimum size=2.5mm, fill=white, draw=black]
\tikzstyle{hedge}=[ultra thick]

\def\irad{0.8cm}
\def\orad{1.9cm}

\foreach \i in {0,1,2}{
  \draw[hedge] (90+\i*120:\irad) -- (210+\i*120:\irad);
  \draw[hedge] (90+\i*120:\irad) -- (90+\i*120:\orad);
  \draw[hedge,red] (60+\i*120:\orad) -- (90+\i*120:\orad);
  \draw[hedge] (90+\i*120:\orad) -- (120+\i*120:\orad);
}

\foreach \i in {0,1,2}{
  \coordinate (v\i) at (90+\i*120:\irad);
  \coordinate (u\i) at (90+\i*120:\orad);
  \coordinate (r\i) at (60+\i*120:\orad);
  \coordinate (l\i) at (120+\i*120:\orad);
}

\node[hvertex,label=above:$u_{i}$] at (u0) {};
\node[hvertex,label=left:$u_{2k+i}$] at (u1) {};
\node[hvertex,label=right:$u_{k+i}$] at (u2) {};

\node[hvertex,label=right:$v_{i}$] at (v0) {};
\node[hvertex,label=below:$\,\,v_{2k+i}$] at (v1) {};
\node[hvertex,label=below:$v_{k+i}$] at (v2) {};

\foreach \i in {1,2}{
  \node[hvertex] at (l\i){};
  \node[hvertex] at (r\i){};
}

\node[hvertex,label=above:$u_{i-1}$] at (l0){};
\node[hvertex,label=above:$u_{i+1}$] at (r0){};

\end{tikzpicture}
\label{Fig_Vertices_V_i}
\caption{A small part of $GP(3k,k)$; the colours of the red edges make up $\phi_i$}
\end{figure}

Our objective is to count 
 the number of $1$-factorisations of $GP$. Rather than counting them directly, we will consider edge colourings,
and here we will see that it suffices to focus on certain edge colourings of the outer cycle.

Let $\phi$ be an edge colouring with colours $\{ \mathsf{1,2,3}\}$ of either the whole 
graph $GP$ or only of the outer cycle $C_O$. 
We split  $\phi$
 into $k$ triples
\[
\phi_i=
\left( 
\phi\left( u_{i}u_{i+1} \right) , 
\phi\left( u_{k+i}u_{k+i+1}\right) ,
\phi\left( u_{2k+i}u_{2k+i+1}\right) 
\right) \text{ for }i=1,\ldots, k.
\]
To keep notation simple, we will omit the parentheses and commas, and only 
write $\phi_i=\mathsf{123}$ to mean $\phi_i=(\mathsf{1,2,3})$. 
We, furthermore, define also $\phi_{k+1}=
\left( 
\phi\left( u_{k+1}u_{k+2} \right) , 
\phi\left( u_{2k+1}u_{2k+2}\right) ,
\phi\left( u_{1}u_{2}\right) 
\right)
$, and note that $\phi_{k+1}$ is obtained from $\phi_1$ by a cyclic shift.

It turns out that the colours on the outer cycle already uniquely determine
the edge colouring on the whole graph. Moreover, it is easy to describe
which colourings of the outer cycle extend to the rest of the graph: 

\begin{lemma} \label{Lem_Comb_colourTriples}
Let $\phi:E(C_O)\to\{\mathsf{1,2,3}\}$ be an edge colouring of $C_O$. Then the following two statements
are equivalent:
\begin{enumerate}[\rm (i)]
\item 
there is an edge colouring $\gamma$ of 
$GP$ with $\gamma|_{C_O}=\phi$; \label{Enum_ExistsGamma}
and
\item there is a permutation $(\mathsf a,\mathsf b,\mathsf c)$ of $(\mathsf{1,2,3})$
so that $\phi_{i}$ and $\phi_{i+1}$ are for all $i=1,\ldots, k$ adjacent vertices 
in one of the graphs $T$ and $H$ in Figure~\ref{Fig_T_and_H}. \label{Enum_AdjacentVert}
\end{enumerate}  
Furthermore, if there is  an edge colouring $\gamma$ of $GP$ as in \eqref{Enum_ExistsGamma} then it is unique. 
\end{lemma}

\begin{figure}[ht] 
\centering
\begin{tikzpicture}[every edge quotes/.style={},scale=1]
\tikzstyle{hvertex}=[thick,circle,inner sep=0.05cm, minimum size=2.5mm, fill=white, draw=black]
\tikzstyle{hedge}=[ultra thick]

\begin{scope}
\def\radius{1.cm}

\foreach \i in {1,2,3}{
    \draw[hedge] (90+\i*120:\radius) 
to node {} (210+\i*120:\radius);
}
\node[hvertex] (123) at (90:\radius) {$\mathsf{abc}$};
\node[hvertex] (312) at (90+120:\radius) {$\mathsf{bca}$};
\node[hvertex] (231) at (90+240:\radius) {$\mathsf{cab}$};
\node at (-1,1.5) {$T$};

\end{scope}

\begin{scope}[shift={(5,0.3)}]

\def\radius{1.2cm}

\foreach \i in {0,...,5}{
    \draw[hedge] (\i*60:\radius) 
to node {} (60+\i*60:\radius);
}
\node[hvertex] (211) at (0:\radius) {$\mathsf{aba}$};
\node[hvertex] (323) at (60:\radius) {$\mathsf{bcc}$};
\node[hvertex] (112) at (120:\radius) {$\mathsf{aab}$};
\node[hvertex] (233) at (180:\radius) {$\mathsf{cbc}$};
\node[hvertex] (121) at (240:\radius) {$\mathsf{baa}$};
\node[hvertex] (332) at (300:\radius) {$\mathsf{ccb}$};

\node at (-1.6,1.2) {$H$};

\end{scope}
\end{tikzpicture}
\caption{The graphs $T$ and $H$ capture the possible combinations of consecutive colour triples}
\label{Fig_T_and_H}
\end{figure}
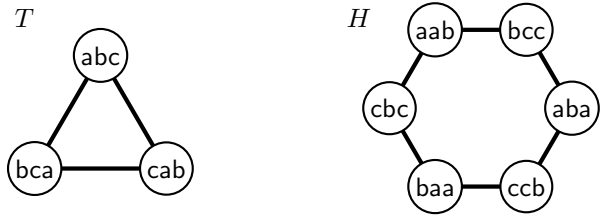

\begin{proof} 
First assume \eqref{Enum_ExistsGamma}, that is, that there is an edge colouring $\gamma$ of $GP$
with $\gamma|_{C_O}=\phi$. Note  that for all $i\in\mathbb Z_{3k}$ 
\begin{equation}\label{spokeeq}
\emtext{the spokes $u_iv_i$, $u_{i+k}v_{i+k}$ and $u_{i+2k}v_{i+2k}$
receive distinct colours.
}
\end{equation}
Indeed, the three edges of the triangle $v_iv_{i+k}v_{i+2k}v_i$
need to be assigned three different colours, which then also must be the 
case for the corresponding spokes. 

From~\eqref{spokeeq} follows that no $\phi_i$ is monochromatic, 
i.e.,that $\phi_i\notin\{ \mathsf{111,222,333}\}$. If there was such an $\phi_i$, say
$\phi_i=\mathsf{111}$, then 
none of the spokes $u_iv_i$, $u_{i+k}v_{i+k}$ and $u_{i+2k}v_{i+2k}$
could be coloured with~$\mathsf 1$ under $\gamma$. 

In particular, $\phi_1$ will contain at least two distinct colours 
and we can choose distinct $\mathsf a$, $\mathsf b$ and $\mathsf c$ so 
that $\phi_1\in\{\mathsf{abc},\mathsf{aab},\mathsf{aba},\mathsf{baa}\}$.
Then $\phi_1$ is a vertex of either $T$ or $H$.
Consider inductively $\phi_i$ to be such a vertex as well. 
By rotational  symmetry  of $GP$ and permutation of colours, we may assume 
that $\phi_i\in\{\mathsf{abc},\mathsf{aab}\}$. 

If $\phi_i=\{\mathsf{abc}\}$ then, by~\eqref{spokeeq}, the spokes 
$u_iv_i$, $u_{i+k}v_{i+k}$ and $u_{i+2k}v_{i+2k}$ can only be coloured
$\mathsf b,\mathsf c,\mathsf a$ (in that order) or 
$\mathsf c,\mathsf a,\mathsf b$. In the first case, the colour of the edge
$u_{i+1}u_{i+2}$ needs to be $\mathsf c$, and so on, resulting in $\phi_{i+1}=\mathsf{cab}$.
In the other case, we get $\phi_{i+1}=\mathsf{bca}$. Both of these
colour triples are adjacent to $\phi_i$ in~$T$.
The proof for $\phi_i=\mathsf{aab}$ is similar. 

\medskip
For the converse direction suppose now that \eqref{Enum_AdjacentVert} holds. Consider a pair
of colour triples $\phi_i$ and $\phi_{i+1}$. By rotation symmetry
of $GP$ and by symmetry of the three colours, we only need to check the cases that 
\begin{equation*}
(\phi_i,\phi_{i+1})=(\mathsf{abc},\mathsf{bca})\text{ and }
(\phi_i,\phi_{i+1})=(\mathsf{aab},\mathsf{cbc}).
\end{equation*}
In the first case, the spokes $u_iv_i$, $u_{i+k}v_{i+k}$ and $u_{i+2k}v_{i+2k}$
can be coloured with $\mathsf c$, $\mathsf a$ and $\mathsf b$ (in that order),
which then permits to colour the triangle $v_iv_{i+k}v_{i+2k}v_i$ with 
$\mathsf{abc}$. Observe that neither for the spokes nor for the triangle 
there was an alternative colouring. For the other case, colour the spokes 
with $\mathsf{bca}$ and then the triangle accordingly. Again, all the colours
are forced. Extending the colouring $\phi$ in this way for all $i$ yields 
an edge colouring $\gamma$ of all of $GP$, and as all colours are forced,
$\gamma$ is uniquely determined by $\phi$. 
\end{proof}

The lemma implies that any edge colouring $\gamma$ of $GP$ corresponds 
to a walk $\gamma_1\gamma_2\ldots\gamma_{k+1}$ 
of length~$k$ in either $T$ or in $H$. Where does such a walk start 
and end? By symmetry, we may  assume that the walk starts at $\gamma_1=\mathsf{abc}$
or $\gamma_1=\mathsf{aab}$. 
 It then ends in $\gamma_{k+1}$, which is either
$\mathsf{bca}$ or $\mathsf{aba}$. Conversely, all such walks define edge colourings
of $GP$.

To count the number of these walks, consider two vertices $x,y$ of $T$, respectively
of $H$, that are at distance~$\ell$ from each other in $T$ (resp.\ in $H$). 
We define 
\begin{align*}
t_k(\ell)  &\coloneqq \sharp \left\lbrace \text{walks of length } k \text{ between $x$ and $y$ in $T$} 
\right\rbrace \\
h_k(\ell)  &\coloneqq \sharp \left\lbrace \text{walks of length } k \text{ between } x \text{ and } y \text{ in } H \right\rbrace 
\end{align*} 
Then every edge colouring of $GP$ corresponds to a walk that is either counted in $t_k(1)$ 
(as $\mathsf{abc}$ and $\mathsf{bca}$ have distance~$1$ in $T$) or counted in $h_k(2)$.

\begin{lemma} \label{Lemma_Number_1factors_walks}
The number of $1$-factorisations of $GP(3k,k)$ equals $t_k(1) + 3 h_k(2)$.
\end{lemma}

\begin{proof}
First, we note that there is a bijection between the $1$-factorisations and 
the edge colourings $\gamma:E\to\{\mathsf{1,2,3}\}$, where $u_1u_2$ is coloured with~$\mathsf{1}$, 
$u_{1+k}u_{2+k}$ coloured with~$\mathsf{1}$ or~$\mathsf{2}$, and $u_{1+2k}u_{2+2k}$ is only coloured
with~$\mathsf{3}$ if $u_{1+k}u_{2+k}$ was coloured with~$\mathsf{2}$. 
By Lemma~\ref{Lem_Comb_colourTriples}, the number of such $\gamma$ is equal to the 
number of edge colourings $\phi:E(C_O)\to\{\mathsf{1,2,3}\}$ satisfying~\eqref{Enum_AdjacentVert} of Lemma~\ref{Lem_Comb_colourTriples} and for which $\phi_1\in\{\mathsf{123,112,121,211}\}$.

How many such edge colourings $\phi$ are there with $\phi_1=\mathsf{123}$? 
Since $\phi_{k+1}=\mathsf{231}$, Lemma~\ref{Lem_Comb_colourTriples} implies that this number 
is $t_k(1)$. Each of the numbers of edge colourings $\phi$ with $\phi_1\in\{\mathsf{112,121,211}\}$
is equal to $h_2(2)$, which means that, in total, we get $t_k(1)+3h_k(2)$ edge colourings.
\end{proof}

We need the closed expression of the Jacobsthal numbers:
\begin{align} \label{Eq_Jacob_Recur}
J(k)&=\tfrac{1}{3} \left(2^{k} +(-1)^{k+1}\right) &\text{ for every } k\geq 0.
\end{align}

\begin{lemma} \label{Lemma_Number_walks_C3}
For any $k$
\[
t_k(0)= \tfrac{1}{3}( 2^k  + 2 (-1)^k)
\text{ and }
t_k(1)  = J(k)
\]
\end{lemma}

The second equation can be found in \cite{Barry07}. Since it follows directly from the first one, 
we will still include a proof.

\begin{proof}
A classic question in algebraic graph theory is to count the number of
pairs $(W,v)$ for a graph $G$, where $W$ is a closed walk of length~$k$ in a graph~$G$ and $v$
the first vertex of $W$. It turns out, see for instance~\cite[Section 1.4]{BeiWil}, 
that this number is equal to $\lambda_1^k + \cdots + \lambda_n^k$, 
where $\lambda_1,\ldots,  \lambda_n$ are the eigenvalues of 
the adjacency matrix of $G$. For the triangle 
these eigenvalues are $2$ and twice $-1$ (see e.g.\ \cite[Section 1.2]{BeiWil}). 
Our aim, however, is to count the number of closed walks of length~$k$ that start
at a specific vertex, which means that we have to divide by $3$. 
This gives $t_k(0)= \tfrac{1}{3}( 2^k  + 2 (-1)^k)$
 for every $k$, and thus the first assertion of the lemma.

Since every walk on $k+1$ edges has to visit a vertex adjacent to the end vertex after the $k$-th step,
for which there are two possibilities, we obtain $t_{k+1}(0)= 2t_k(1)$ and thus
\begin{align*}
2t_k(1)  =t_{k+1}(0)= \tfrac{1}{3} \left(2^{k+1} +2 (-1)^{k+1}\right)
\end{align*}
By~\eqref{Eq_Jacob_Recur}, we get $2t_k(1)=2J(k)$ for every $k$.
\end{proof}

For the next lemma, we label the vertices of $T$ as $x_0, x_1, x_2$ in clockwise order. 
The vertices of $H$ are  $y_0, y_1, y_2, z_0, z_1, z_2$ in clockwise order. 

\begin{lemma} \label{Lemma_Bijection_undir_walks}
If $k$ is even, there is a bijection between the set of walks of length $k$ from $x_0$ to $x_2$ in $T$ and the set of walks from $y_0$ to $y_2$ in $H$.
Moreover, 
the $\ell$th edge of the walk in $T$ is traversed in clockwise direction if and only if 
this is also the case for $\ell$th edge of the corresponding walk in $H$.
\end{lemma}

The projection of the set $\lbrace  y_i, z_i \rbrace$ in $H$ to the set $\lbrace  x_i \rbrace$ in $T$ for $i \in \Zz_3$ yields a covering map. The bijection between the considered walks of length $k$ in $T$ and $H$ follows immediately by the path lifting property for covering spaces (see e.g.\ \cite[Section 1.3]{Hatcher}). 
We  give, nevertheless, an elementary proof of the lemma.

\begin{proof}[Proof of Lemma~\ref{Lemma_Bijection_undir_walks}]
In $H$, every vertex with index $i \in \Zz_3$ in $H$ is adjacent to exactly one vertex with index $i+1$ in clockwise direction and one with index $i-1$ in counter-clockwise direction. Therefore, the following rule translates a walk $W=w_0\ldots w_k$ in $T$ starting in $x_0$
to a walk $W'=w'_0\ldots w'_k$ in $H$ starting in $y_0$ while maintaining the directions of edge traversals:
let $w'_0=y_0$;
for $\ell=1,\ldots,k$,
if $w_{\ell}=x_i$ then pick $w'_\ell$ to be the neighbour
of $w'_{\ell-1}$ among $y_i,z_i$. 

As $w_{k+1}=x_2$, the last vertex $w'_{k+1}$ of $W'$ has to be one of  $y_2, z_2$.
Since $k$ is even and the distance between $y_0$ and $z_2$ in $H$ is odd, $W'$ must terminate in~$y_2$. 
\end{proof}

\begin{proof}[Proof of Theorem~\ref{Thm_Number_1fact}]
By Lemma~\ref{Lemma_Number_1factors_walks} the number of $1$-factorisations of $GP(3k,k)$ is $t_k(1) +3 h_k(2)$. 
Lemma~\ref{Lemma_Bijection_undir_walks} yields that $h_k(2)$ equals $t_k(1)$ for even $k$. Since there is no walk of odd length in $H$ that connects two vertices of even distance, $h_k(2)$ is zero for odd $k$. By Lemma~\ref{Lemma_Number_walks_C3} the number of $1$-factorisations now equals $J(k) + 3J(k)$ if $k$ is even and $J(k) +3\cdot 0$ otherwise.
\end{proof}

\section{List edge colouring}

In order to show the list edge conjecture for $GP=GP(3k,k)$, we will 
use the method of Alon and Tarsi \cite{A_Tar92}, 
or rather its specialisation to regular graphs \cite{EG96}.

To define a local rotation, we consider $GP(3k,k)$ always to be drawn as  in Figure~\ref{duererfig}: 
the vertices $u_i$ for $i=1, \ldots, 3k$ are placed on an outer circle in clockwise order, 
the vertices $v_i$ for $i=1, \ldots, 3k$ on a smaller concentric circle in such a way 
that $u_i$ and $v_i$ match up, and all edges are straight.
We define the sign of $\gamma$ at a vertex $w$ as $+$ if the colours $\mathsf{1}, \mathsf{2}, \mathsf{3}$ appear in clockwise order on the incident edges; otherwise the sign is $-$.
More formally, 
\begin{align*}
\sign_{\gamma}(u_i) &=
\begin{cases}
+  & \text{ if }\left( \gamma(u_{i-1}u_i), \gamma(u_i u_{i+1}), \gamma(u_iv_i) \right) \in \lbrace \mathsf{123} ,   \mathsf{231} , \mathsf{312} \rbrace \\
-  & \text{ otherwise}.
\end{cases} \\
\sign_{\gamma}(v_i) &=
\begin{cases}
+  & \text{ if }\left( \gamma(v_{k+i}v_i), \gamma(v_{i}v_{2k+i}), \gamma(v_iu_i) \right) \in \lbrace \mathsf{123} ,   \mathsf{231}, \mathsf{312} \rbrace \\
-  & \text{ otherwise}
\end{cases} 
\end{align*}
The sign of the colouring $\gamma$ is then
\begin{equation*} 
\sign(\gamma) \coloneqq \prod\limits_{v \in V(GP)} \sign_{\gamma}(v).
\end{equation*}

Permuting colours in our context does not change the 
sign of an edge colouring. This is true in all regular graphs, see for instance~\cite{EG96}:

\begin{lemma}\label{colourpermlem}
Let $G$ be a $d$-regular graph, and let $\gamma$ be an edge colouring of $G$ with
$d$ colours. If $\gamma'$ is obtained from $\gamma$ by exchanging two colours, 
then $\sign(\gamma)=\sign(\gamma')$. 
\end{lemma}

For $d$-regular graphs with odd $d$, such as cubic graphs, 
Lemma~\ref{colourpermlem} is easy to see: the signs of $\gamma$ and $\gamma'$ differ at every vertex of $G$,
and there is an even number of vertices in total.

Lemma~\ref{colourpermlem} allows to define a sign $\sign f$ for any $1$-factorisation $f$ 
by fixing it to the sign 
of any edge colouring that induces $f$. The Alon-Tarsi colouring criterion now takes a particularly simple form in $d$-regular graphs; see Ellingham and Goddyn~\cite{EG96} or Alon~\cite{Alo93}.

\begin{theorem} \label{Thm_ElGo}
Let $G$ be a $d$-regular graph with 
\begin{equation*}
\sum\limits_{f \, 1\text{-factor of }G} \sign(f) \not= 0.
\end{equation*} 
Then, $G$ is $d$-list-edge-colourable.
\end{theorem}

Applying Theorem~\ref{Thm_ElGo} to  $GP(3k,k)$ with odd $k$, 
we can now see that the list edge-colouring conjecture holds:

\begin{corollary} \label{Cor_Number_pos_neg_1F_k_odd}
For odd $k$, the  graph $GP(3k,k)$ has choice index $3$.
\end{corollary}

\begin{proof}
By Theorem~\ref{Thm_Number_1fact}, $GP(3k,k)$ has $J(k)= \tfrac{2^k +1}{3}$  distinct $1$-factorisations, 
if $k$ is odd. 
Since this number is odd, the sum of the signs of all $1$-factorisations cannot be zero. Theorem~\ref{Thm_ElGo} finishes the proof.
\end{proof}

Unfortunately, for even $k$ 
the number of $1$-factorisations is even. 
That means, we have to put a bit more effort into showing that the sum of the signs of all $1$-factorisations is not zero. In particular, we will need to count the 
positive and negative $1$-factorisations separately.

As a first step, we refine the colour triple graphs $T$ and $H$, and endow 
them with signs on the edges. Figure~\ref{Fig_Tpm_and_Hpm}
shows the graphs $T_\pm$ and $H_\pm$, which we obtain from $T$ and $H$ by 
replacing each edge by two inverse directed edges, each having a sign. 
Note that in $T_\pm$ all edges in clockwise direction are positive, 
while clockwise edges in $H_\pm$ are negative. 

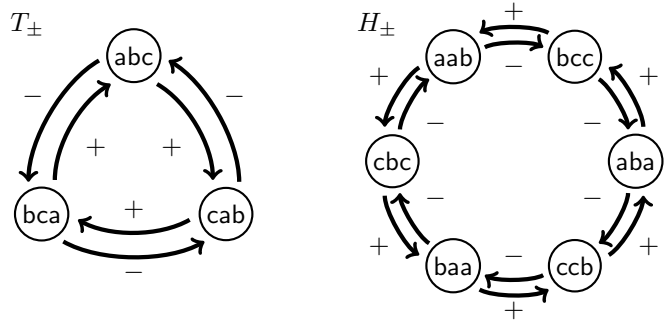
\begin{figure}[ht]
\centering
\begin{tikzpicture}[every edge quotes/.style={},scale=1]
\tikzstyle{hvertex}=[thick,circle,inner sep=0.05cm, minimum size=2.5mm, fill=white, draw=black]
\tikzstyle{hedge}=[ultra thick]

\begin{scope}
\def\radius{1.4cm}
\def\orad{1.6cm}
\def\irad{1.2cm}

\node[hvertex] (123) at (90:\radius) {$\mathsf{abc}$};
\node[hvertex] (312) at (90+120:\radius) {$\mathsf{bca}$};
\node[hvertex] (231) at (90+240:\radius) {$\mathsf{cab}$};

\foreach \i in {1,2,3}{
    \draw[hedge,->,shorten <=15pt, shorten >=15pt,bend right,auto,swap] (90+\i*120:\orad) 
to node {$\mathbf{-}$} (210+\i*120:\orad);
    \draw[hedge,<-,shorten <=10pt, shorten >=10pt,bend right,auto] (90+\i*120:\irad) 
to node {$\mathbf{+}$} (210+\i*120:\irad);
}
\node at (-1.4,1.8) {$T_\pm$};
\end{scope}

\begin{scope}[shift={(5,0)}]

\def\radius{1.6cm}
\def\orad{1.7cm}
\def\irad{1.5cm}

\node[hvertex] (211) at (0:\radius) {$\mathsf{aba}$};
\node[hvertex] (323) at (60:\radius) {$\mathsf{bcc}$};
\node[hvertex] (112) at (120:\radius) {$\mathsf{aab}$};
\node[hvertex] (233) at (180:\radius) {$\mathsf{cbc}$};
\node[hvertex] (121) at (240:\radius) {$\mathsf{baa}$};
\node[hvertex] (332) at (300:\radius) {$\mathsf{ccb}$};

\foreach \i in {0,...,5}{
    \draw[hedge,->,shorten <=13pt, shorten >=13pt,bend right,auto,swap] (\i*60:\orad) 
to node {$\mathbf{+}$} (60+\i*60:\orad);
    \draw[hedge,<-,shorten <=12pt, shorten >=12pt,bend right,auto] (\i*60:\irad) 
to node {$\mathbf{-}$} (60+\i*60:\irad);
}
\node at (-1.8,1.8) {$H_\pm$};

\end{scope}
\end{tikzpicture}
\caption{Signs of the possible combinations of consecutive colour triples}
\label{Fig_Tpm_and_Hpm}
\end{figure}

Let $x,y$ be two adjacent vertices in $T_\pm$ or in $H_\pm$. We denote
the \emph{sign  of the edge} pointing from $x$ to $y$ by $\sign(x,y)$. 
The next lemma shows that the signs on the edges capture the signs of 
 edge colourings. 

\begin{lemma} \label{Lemma_signs_sigma_i}
Let $\gamma:E(GP)\to \{\mathsf{1,2,3}\}$ be an edge colouring of $GP$, 
and let $(\mathsf a,\mathsf b,\mathsf c)$ be a 
permutation of  $(\mathsf{1,2,3})$ so that $\gamma_1$ is a vertex in $T_\pm$
or in $H_\pm$. Then
\[
\sign(\gamma)=\prod_{i=1}^k \sign(\gamma_i,\gamma_{i+1}).
\]
\end{lemma}

\begin{proof}
We partition the vertices of $GP$ into $k$ parts, namely into the sets 
\begin{equation*}
V_{i-1} \coloneqq \lbrace u_{i},v_{i}, u_{k+i},v_{k+i}, u_{2k+i},v_{2k+i} \rbrace  \text{ for } i \in \Zz_k.
\end{equation*}
See Figure~\ref{Fig_Vertices_V_i} for the vertices in $V_{i-1}$.
Factorising
\begin{equation*}
\sign(\gamma) 
= \prod\limits_{w \in V} \sign_{\gamma}(w) 
= \prod\limits_{i \in \Zz_k} \prod\limits_{w \in V_i} \sign_{\gamma}(w),
\end{equation*}
we see that the lemma is proved if 
\begin{equation}\label{Viclaim}
\prod\limits_{w \in V_i} \sign_{\gamma}(w)=\sign(\gamma_i,\gamma_{i+1})
\end{equation}
holds true for all~$i=1, \ldots, k$.

That the total sign on $V_i$ depends only on $\gamma_i$ and $\gamma_{i+1}$ is clear 
from Lemma~\ref{Lem_Comb_colourTriples}: $\gamma_i$ and $\gamma_{i+1}$
determine the colours of the edges incident with vertices in $V_i$.
Therefore, there is a function $f$ on the edges of $T_\pm\cup H_\pm$ 
to $\{-,+\}$
so that 
\[
\prod\limits_{w \in V_i} \sign_{\gamma}(w)=f(\gamma_i,\gamma_{i+1})
\] 
Our task reduces to verifying that $f(\gamma_i,\gamma_{i+1})=\sign(\gamma_i,\gamma_{i+1})$.
In principle, we could now check all edges in $T_\pm$ and $H_\pm$, one by one, 
to see whether the signs are correct. 
Instead, we exploit the fact that all vertices in $T_\pm$ (or in $H_\pm$) are in 
some sense the same. 

\begin{figure}[ht]
\centering
\begin{tikzpicture}[scale=0.8,auto]
\tikzstyle{hvertex}=[thick,circle,inner sep=0.cm, minimum size=2.5mm, fill=white, draw=black]
\tikzstyle{hedge}=[ultra thick]
\tikzstyle{redge}=[ultra thick,red]
\tikzstyle{bedge}=[ultra thick,blue]

\def\irad{0.8cm}
\def\orad{1.9cm}
\def\labelpos{150:1.5*\orad}

\node[hvertex] (v1) at (90:\irad) {$-$};
\node[hvertex] (v3) at (90+120:\irad) {$-$};
\node[hvertex] (v2) at (90+240:\irad) {$-$};

\node[hvertex] (u1) at (90:\orad) {$+$};
\node[hvertex] (u3) at (90+120:\orad) {$+$};
\node[hvertex] (u2) at (90+240:\orad) {$+$};

\draw[hedge] (v1) to node[pos=0.6]{$\mathsf b$} (v2);
\draw[hedge] (v2) to node[pos=0.5]{$\mathsf c$} (v3);
\draw[hedge] (v3) to node[pos=0.4]{$\mathsf a$} (v1);

\draw[hedge] (v1) to node[pos=0.5]{$\mathsf c$} (u1);
\draw[hedge] (v2) to node[pos=0.5,swap]{$\mathsf a$} (u2);
\draw[hedge] (v3) to node[pos=0.5]{$\mathsf b$} (u3);

\coordinate (r1) at (60:\orad);
\coordinate (l1) at (120:\orad);
\coordinate (r3) at (60+120:\orad);
\coordinate (l3) at (120+120:\orad);
\coordinate (r2) at (60+240:\orad);
\coordinate (l2) at (120+240:\orad);

\draw[redge] (l1) to node[pos=0.5]{$\mathsf a$} (u1);
\draw[bedge] (u1) to node[pos=0.5]{$\mathsf b$} (r1);
\draw[redge] (l2) to node[pos=0.3]{$\mathsf b$} (u2);
\draw[bedge] (u2) to node[pos=0.5]{$\mathsf c$} (r2);
\draw[redge] (l3) to node[pos=0.5]{$\mathsf c$} (u3);
\draw[bedge] (u3) to node[pos=0.7]{$\mathsf a$} (r3);

\node at (\labelpos){$\mathsf{abc\to bca}$};

\begin{scope}[shift={(6.5,0)}]

\def\irad{0.8cm}
\def\orad{1.9cm}

\node[hvertex] (v1) at (90:\irad) {$-$};
\node[hvertex] (v3) at (90+120:\irad) {$-$};
\node[hvertex] (v2) at (90+240:\irad) {$-$};

\node[hvertex] (u1) at (90:\orad) {$-$};
\node[hvertex] (u3) at (90+120:\orad) {$-$};
\node[hvertex] (u2) at (90+240:\orad) {$-$};

\draw[hedge] (v1) to node[pos=0.6]{$\mathsf b$} (v2);
\draw[hedge] (v2) to node[pos=0.5]{$\mathsf c$} (v3);
\draw[hedge] (v3) to node[pos=0.4]{$\mathsf a$} (v1);

\draw[hedge] (v1) to node[pos=0.5]{$\mathsf c$} (u1);
\draw[hedge] (v2) to node[pos=0.5,swap]{$\mathsf a$} (u2);
\draw[hedge] (v3) to node[pos=0.5]{$\mathsf b$} (u3);

\coordinate (r1) at (60:\orad);
\coordinate (l1) at (120:\orad);
\coordinate (r3) at (60+120:\orad);
\coordinate (l3) at (120+120:\orad);
\coordinate (r2) at (60+240:\orad);
\coordinate (l2) at (120+240:\orad);

\draw[redge] (l1) to node[pos=0.5]{$\mathsf b$} (u1);
\draw[bedge] (u1) to node[pos=0.5]{$\mathsf a$} (r1);
\draw[redge] (l2) to node[pos=0.3]{$\mathsf c$} (u2);
\draw[bedge] (u2) to node[pos=0.5]{$\mathsf b$} (r2);
\draw[redge] (l3) to node[pos=0.5]{$\mathsf a$} (u3);
\draw[bedge] (u3) to node[pos=0.7]{$\mathsf c$} (r3);

\node at (\labelpos){$\mathsf{bca\to abc}$};

\end{scope}

\begin{scope}[shift={(0,-4.7)}]
\node[hvertex] (v1) at (90:\irad) {$+$};
\node[hvertex] (v3) at (90+120:\irad) {$+$};
\node[hvertex] (v2) at (90+240:\irad) {$+$};

\node[hvertex] (u1) at (90:\orad) {$+$};
\node[hvertex] (u3) at (90+120:\orad) {$+$};
\node[hvertex] (u2) at (90+240:\orad) {$-$};

\draw[hedge] (v1) to node[pos=0.6]{$\mathsf a$} (v2);
\draw[hedge] (v2) to node[pos=0.5]{$\mathsf c$} (v3);
\draw[hedge] (v3) to node[pos=0.4]{$\mathsf b$} (v1);

\draw[hedge] (v1) to node[pos=0.5]{$\mathsf c$} (u1);
\draw[hedge] (v2) to node[pos=0.5,swap]{$\mathsf b$} (u2);
\draw[hedge] (v3) to node[pos=0.5]{$\mathsf a$} (u3);

\coordinate (r1) at (60:\orad);
\coordinate (l1) at (120:\orad);
\coordinate (r3) at (60+120:\orad);
\coordinate (l3) at (120+120:\orad);
\coordinate (r2) at (60+240:\orad);
\coordinate (l2) at (120+240:\orad);

\draw[redge] (l1) to node[pos=0.5]{$\mathsf a$} (u1);
\draw[bedge] (u1) to node[pos=0.5]{$\mathsf b$} (r1);
\draw[redge] (l2) to node[pos=0.3]{$\mathsf a$} (u2);
\draw[bedge] (u2) to node[pos=0.5]{$\mathsf c$} (r2);
\draw[redge] (l3) to node[pos=0.5]{$\mathsf b$} (u3);
\draw[bedge] (u3) to node[pos=0.7]{$\mathsf c$} (r3);

\node at (\labelpos){$\mathsf{aab\to bcc}$};
\end{scope}

\begin{scope}[shift={(6.5,-4.7)}]

\def\irad{0.8cm}
\def\orad{1.9cm}

\node[hvertex] (v1) at (90:\irad) {$+$};
\node[hvertex] (v3) at (90+120:\irad) {$+$};
\node[hvertex] (v2) at (90+240:\irad) {$+$};

\node[hvertex] (u1) at (90:\orad) {$-$};
\node[hvertex] (u3) at (90+120:\orad) {$-$};
\node[hvertex] (u2) at (90+240:\orad) {$+$};

\draw[hedge] (v1) to node[pos=0.6]{$\mathsf a$} (v2);
\draw[hedge] (v2) to node[pos=0.5]{$\mathsf c$} (v3);
\draw[hedge] (v3) to node[pos=0.4]{$\mathsf b$} (v1);

\draw[hedge] (v1) to node[pos=0.5]{$\mathsf c$} (u1);
\draw[hedge] (v2) to node[pos=0.5,swap]{$\mathsf b$} (u2);
\draw[hedge] (v3) to node[pos=0.5]{$\mathsf a$} (u3);

\coordinate (r1) at (60:\orad);
\coordinate (l1) at (120:\orad);
\coordinate (r3) at (60+120:\orad);
\coordinate (l3) at (120+120:\orad);
\coordinate (r2) at (60+240:\orad);
\coordinate (l2) at (120+240:\orad);

\draw[redge] (l1) to node[pos=0.5]{$\mathsf b$} (u1);
\draw[bedge] (u1) to node[pos=0.5]{$\mathsf a$} (r1);
\draw[redge] (l2) to node[pos=0.3]{$\mathsf c$} (u2);
\draw[bedge] (u2) to node[pos=0.5]{$\mathsf a$} (r2);
\draw[redge] (l3) to node[pos=0.5]{$\mathsf c$} (u3);
\draw[bedge] (u3) to node[pos=0.7]{$\mathsf b$} (r3);

\node at (\labelpos){$\mathsf{bcc\to aab}$};

\end{scope}

\end{tikzpicture}
\caption{Signs of the vertices in $V_i$ for some consecutive colour triples}\label{signsfig}
\end{figure}
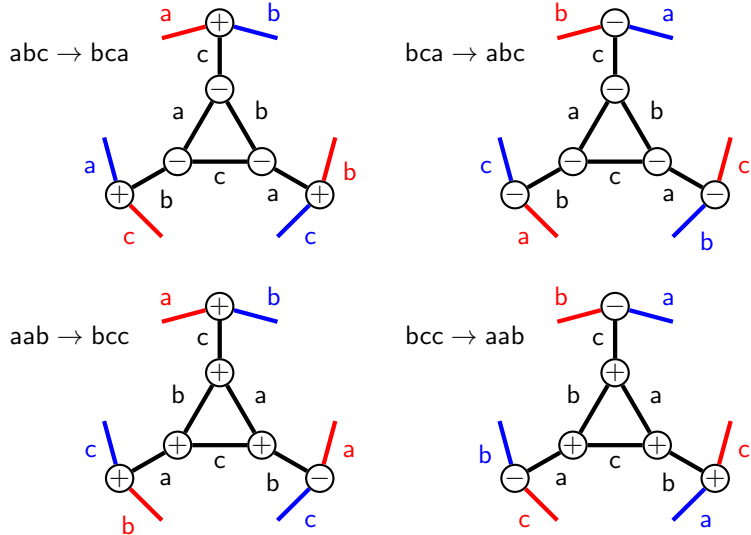

A clockwise rotation of $GP$ by $k$ vertices  induces a shift in a colour triple
$\gamma_i$ from $(\gamma_{i1},\gamma_{i2},\gamma_{i3})$ to $(\gamma_{i2},\gamma_{i3},\gamma_{i1})$.
Note that a rotation of $GP$ obviously does not change the sign of $\gamma$. 
Moreover, permutation of colours preserves the total sign of $V_i$ 
 since swapping two colours changes the sign at all six vertices. 
Therefore we may assume that $\{\gamma_i,\gamma_{i+1}\}=\{\mathsf{abc}, \mathsf{bca}\}$
(if $\gamma_1\in T_\pm$) or that $\{\gamma_i,\gamma_{i+1}\}=\{\mathsf{aab}, \mathsf{bcc}\}$
(if $\gamma_1\in H_\pm$).

This gives four constellations to check, as the sign can (and does) depend on the direction 
of the edge from $\gamma_i$ to $\gamma_{i+1}$. The four constellations are shown 
in Figure~\ref{signsfig}, where we can see, for instance, that the edge from $\mathsf{abc}$
to $\mathsf{bca}$ has a net negative sign under~$f$, while the inverse edge is positive. 

Since, on these four constellations, $f$ coincides with the edge signs of $T_\pm$ and $H_\pm$,
it coincides everywhere, which proves~\eqref{Viclaim}.
\end{proof}

Lemmas~\ref{Lem_Comb_colourTriples} and~\ref{Lemma_signs_sigma_i} 
imply that every positive $1$-factorisation corresponds to a walk in either $T_{\pm}$ or $H_{\pm}$ whose edge signs multiply to $+$. 
We  call such a walk \emph{positive}; whereas a walk whose signs multiply to $-$ is \emph{negative}.

To count such walks, we observe that a walk and its reverse walk might have different signs. Not only the distance between two vertices has an influence, but also the rotational direction of the shortest path. 

For two vertices $x$ and $y$ for which
the clockwise path from $x$ to $y$ is of length $\ell$, we define
\begin{align*}
t_k^{+}(\ell)  &\coloneqq \sharp \left\lbrace \text{ positive walks of length } k \text{ from } x \text{ to } y \text{ in } T_\pm \right\rbrace \\
h_k^{+}(\ell)  &\coloneqq \sharp \left\lbrace \text{ positive walks of length } k \text{ from } x \text{ to } y \text{ in } H_\pm \right\rbrace
\end{align*} 
and $t_k^-(\ell)$ and $ h_k^-(\ell)$ analogously.

Similarly as in Section~\ref{Section_Counting_1-fact} for unsigned colourings, every positive edge colouring of $GP$ now corresponds to a positive walk in $T_{\pm}$ or in $H_{\pm}$. Since all edge colourings with the same associated $1$-factorisation have the same sign, we thus have 
a way to count positive and negative $1$-factorisations via walks in signed graphs:

\begin{lemma} \label{Lem_Number_pos_neg_1F_walks}
The number of positive/negative $1$-factorisations of $GP(3k,k)$ 
is equal to $t^{\pm}_k(2) + 3 h^{\pm}_k(2)$.
\end{lemma}

\begin{proof}
As  before, in order to count $1$-factorisations it suffices to 
count edge colourings $\gamma$ with $\gamma_1\in\{\mathsf{123},\mathsf{112}, \mathsf{121}, \mathsf{211}\}$.
Lemma~\ref{Lemma_signs_sigma_i} in conjunction with Lemma~\ref{Lem_Comb_colourTriples}
shows that there is a one-to-one correspondence between  positive (resp.\ negative) edge colourings 
and certain positive (resp.\ negative) walks of length $k$ in $T_{\pm}$ and in $H_{\pm}$. 
Namely, these are the $t^{\pm}_k(2)$ walks in $T_{\pm}$ from $\mathsf{123}$ to $\mathsf{231}$
plus the $3h_k^{\pm}(2)$ walks in $H_{\pm}$ with starting 
point $\mathsf{112}, \mathsf{121}, \mathsf{211}$, and respective 
end point $\mathsf{121}, \mathsf{211}, \mathsf{112}$. 
\end{proof}

As in Lemma~\ref{Lemma_Bijection_undir_walks} we can state a connection between walks in $T_{\pm}$ and $H_{\pm}$.

\begin{lemma} \label{Lemma_Bijection_signed_walks}
$h^{\pm}_{k}(2)=t^{\pm}_{k}(2)$ for even $k$.
\end{lemma}

\begin{proof}
We can canonically extend the map of Lemma~\ref{Lemma_Bijection_undir_walks} to a bijection between walks in $T_{\pm}$ and $H_{\pm}$. Then the bijection maps walks counted by $t_k(1)=t_k⁺(2)+t_k^-(2)$ to walks counted by $h_k(2)=h_k^+(2)+ h_k^-(2)$. 
Since the signs of the clockwise (resp.\ anti-clockwise) arcs are different
in $T_\pm$ and $H_\pm$, any arc in a walk in $T_\pm$ has a different sign
from its image in $H_\pm$. However, as we consider walks of even length~$k$,
 the total sign of the walks is preserved by the bijection, and the assertion follows.
\end{proof}

In order to show that the numbers of positive and negative $1$-factorisations differ, it remains to compute $t_k^+$ and $t_k^-$:
\begin{lemma} \label{Lem_Number_pos_neg_walks_T}
For any integer $k \geq 1$ 
\begin{align*}
t^+_{k}(2) &= \tfrac{1}{6} \left( 2^k - (-1)^{k}\left(1+ (-3)^{\left\lceil \frac{k}{2} \right\rceil }\right)\right)\\
t^-_{k}(2) &= \tfrac{1}{6} \left( 2^k - (-1)^{k}\left(1- (-3)^{\left\lceil \frac{k}{2} \right\rceil }\right)\right)
\end{align*}
\end{lemma}

\begin{proof}
Every walk in $T^{\pm}$ from a vertex $v$ to a vertex $v'$ induces a reflected walk from $v'$ to $v$. In that walk, every arc is replaced by its reversed arc, which has opposite sign. Furthermore, the shortest path from $v$ to $v'$ in clockwise direction is of length $1$ if and only if the shortest path from $v'$ to $v$ in clockwise direction has length $2$. Therefore
\begin{equation}
t_k^{\pm}(1)= 
\begin{cases}
t_k^{\pm}(2) & \text{ if } k \text{ is even } \\
t_k^{\mp}(2) & \text{ if } k \text{ is odd }
\end{cases} 
\label{Eq_Indexswap}
\end{equation}
Note that the signs swap for odd $k$. 

In the same way follows for odd $k$ that $t_k^{+}(0)= t_k^{-}(0 )$. 
Since $t_k^{+}(0) + t_k^{-}(0)=t_k(0)$ we get 
$t_k^\pm(0)=\tfrac{1}{2}t_k(0)$ and thus 
with 
Lemma~\ref{Lemma_Number_walks_C3} that 
\begin{equation}\label{tk0eq}
t_k^\pm(0) = \tfrac{1}{3}(2^{k-1}-1)\quad\text{ for odd }k.
\end{equation}
We use Lemma~\ref{Lemma_Number_walks_C3} together with \eqref{Eq_Jacob_Recur} and note for later that
\begin{equation}
t_k^{+}(\ell) + t_k^{-}(\ell)= t_k(\ell)= 
 J(k) = \tfrac{1}{3}\left(2^k - (-1)^k\right) \quad\text{ for } \ell \in \{ 1,2 \} 
\label{Eq_Sum_+-}
\end{equation}
Trivially, a walk of length $k$ must visit a neighbour of 
its last vertex in the $(k-1)$th step, 
and a vertex adjacent to its penultimate vertex in the $(k-2)$th step. Therefore
\[
\begin{array}{ll}
t_k^{\pm}(2) = t_{k-1}^{\mp}(0) + t_{k-1}^{\pm}(1) &\text{for all } k \geq 1\\[3pt]	
t_k^{\pm}(2) = 2 t_{k-2}^{\mp}(2) + t_{k-2}^{\pm}(0) + t_{k-2}^{\pm}(1) &\text{for all } k \geq 2.
\end{array}
\]
Applying~\eqref{Eq_Indexswap} and~\eqref{tk0eq}, 
we obtain 
\begin{align}
t_k^{\pm}(2) 
&= t_{k-1}^{\mp}(0) + t_{k-1}^{\pm}(1) 
\nonumber \\
&= \tfrac{1}{3}\left(2^{k-2} -1\right) + t_{k-1}^{\mp}(2) \quad  &\text{ if } k \text{ is even.} \label{Eq_even_k} 
\end{align}
For odd $k$ we get a recurrence relation by
using again~\eqref{Eq_Indexswap},~\eqref{tk0eq} and additionally~\eqref{Eq_Sum_+-}
\begin{align}
t_k^{\pm}(2) 
&= 2 t_{k-2}^{\mp}(2) + t_{k-2}^{\pm}(0) + t_{k-2}^{\pm}(1) \nonumber\\
&= 3 t_{k-2}^{\mp}(2) + \tfrac{1}{6} \left(  2^{k-2} -2 \right) \nonumber\\
&= -3 t_{k-2}^{\pm}(2) + \tfrac{1}{6} \left( 6 \left( 2^{k-2} +1 \right) + 2^{k-2} -2\right) \nonumber\\
&= - 3 t_{k-2}^{\pm}(2) + \tfrac{1}{6} \left( 7 \cdot 2^{k-2} + 4 \right)  \quad  &\text{ if } k \text{ is odd.} \label{Eq_odd_k} 
\end{align}
It straightforward to check that 
\begin{equation*}
t^+_{k}(2) 
= \tfrac{1}{6} \left( (-3)^{\tfrac{k+1}{2}} + 2^{k} +1\right)  \quad\text{ for odd } k 
\end{equation*}
satisfies the recurrence relation~\eqref{Eq_odd_k} and the initial condition $t^+_1(2)=0$. 


Therefore, we deduce with~\eqref{Eq_Sum_+-} that 
\begin{align}
t^-_{k}(2) 
&= \tfrac{1}{6}\left(2^{k+1} +2\right) - \tfrac{1}{6} \left( (-3)^{\tfrac{k+1}{2}} + 2^{k} +1\right) \nonumber\\
&= \tfrac{1}{6} \left(  - (-3)^{\tfrac{k+1}{2}} + 2^{k} +1 \right)& \text{ if } k \text{ is odd}. \nonumber
\intertext{The transition to even $k$ is now made by using~\eqref{Eq_even_k}. We obtain }
t_k^{+}(2) 
&= \tfrac{1}{3}\left(2^{k-2} -1\right) + t_{k-1}^{-}(2) \nonumber \\
&= \tfrac{1}{6}\left(2^{k-1} -2\right) + \tfrac{1}{6} \left(- (-3)^{\tfrac{k}{2}}  +  2^{k-1} +1 \right) \nonumber\\
&= \tfrac{1}{6} \left( - (-3)^{\tfrac{k}{2}}  + 2^{k} -1 \right) &\text{ if } k \geq 2 \text{ is even } \nonumber\\
t_k^{-}(2) 
&= \tfrac{1}{3}\left(2^{k-2} -1\right) + t_{k-1}^{+}(2) \nonumber\\
&= \tfrac{1}{6}\left(2^{k-1} -2\right) + \tfrac{1}{6} \left( (-3)^{\tfrac{k}{2}} + 2^{k-1} +1\right)\nonumber\\
&= \tfrac{1}{6} \left( (-3)^{\tfrac{k}{2}}  + 2^{k} -1  \right)  &\text{ if } k\geq 2 \text{ is even}. \nonumber
\end{align}
The different values of $(-1)^k$ and $\left\lceil \tfrac{k}{2} \right\rceil $ for odd and even $k$ yield the formulas for $t_k^{+}$ and $t_k^{-}$.
\end{proof}

We have finally collected all necessary facts to finish the proof of Theorem~\ref{Thm_LECC_proved}.

\begin{proof}[Proof of Theorem~\ref{Thm_LECC_proved}]
By Corollary~\ref{Cor_Number_pos_neg_1F_k_odd}, it remains to consider the case of even $k$. 

By Lemma~\ref{Lem_Number_pos_neg_1F_walks}, the number of positive/negative of $1$-factorisations is equal to $t^{\pm}_{k}(2) + 3 c^{\pm}_{k}(2)$, 
which is the same as $4t^{\pm}_{k}(2)$, by 
Lemma~\ref{Lemma_Bijection_signed_walks}.
Applying Lemma~\ref{Lem_Number_pos_neg_walks_T}, we see that $t_k^+(2)\neq t_k^-(2)$, 
which shows that the sum of the signs of all $1$-factorisations
is not zero. Thus, the Alon-Tarsi criterion, Theorem~\ref{Thm_ElGo}, 
concludes the proof.
\end{proof}

\bibliographystyle{abbrv}

\bibliography{GenPet_Bib.bib}

\vfill

\small
\vskip2mm plus 1fill
\noindent
Version \today{}
\bigbreak

\noindent
Henning Bruhn
{\tt <henning.bruhn@uni-ulm.de>}\\
Laura Gellert
{\tt <laura.gellert@uni-ulm.de>}\\
Jacob G\"unther\\
Institut f\"ur Optimierung und Operations Research\\
Universit\"at Ulm, Ulm\\
Germany\\

\end{document}